\documentclass{amsart}
\usepackage{setspace}
\usepackage{a4}
\usepackage{amsthm}
\usepackage{latexsym}
\usepackage{amsfonts}
\usepackage{graphicx}
\usepackage{textcomp}
\usepackage{cite}
\usepackage{enumerate}
\usepackage{amssymb}
\usepackage{hyperref}
\usepackage{amsmath}
\usepackage{tikz}
\usepackage[mathscr]{euscript}
\usepackage{mathtools}
\newtheorem{theorem}{Theorem}[section]

\newtheorem{corollary}[theorem] {Corollary}

\newtheorem{question}[theorem]{Question}
\title{This is the title}
\usepackage{fancyhdr}

\begin{document}
\hrule\hrule\hrule\hrule\hrule
\vspace{0.3cm}	
\begin{center}
{\bf{NONLINEAR HEISENBERG-ROBERTSON-SCHRODINGER   UNCERTAINTY PRINCIPLE}}\\
\vspace{0.3cm}
\hrule\hrule\hrule\hrule\hrule
\vspace{0.3cm}
\textbf{K. MAHESH KRISHNA}\\
School of Mathematics and Natural Sciences\\
Chanakya University Global Campus\\
NH-648, Haraluru Village\\
Devanahalli Taluk, 	Bengaluru  North District\\
Karnataka State 562 110 India\\
Email: kmaheshak@gmail.com\\

Date: \today
\end{center}

\hrule\hrule
\vspace{0.5cm}
\textbf{Abstract}:  We derive an uncertainty principle for Lipschitz maps acting on subsets of Banach spaces. We show that this nonlinear uncertainty principle reduces to the Heisenberg-Robertson-Schrodinger uncertainty principle for linear operators acting on Hilbert spaces.



\hrule

\hrule
\section{Introduction}
Let $\mathcal{H}$ be a complex Hilbert space and $A$ be a possibly unbounded self-adjoint linear operator defined on domain $\mathcal{D}(A)\subseteq \mathcal{H}$. For $h \in \mathcal{D}(A)$ with $\|h\|=1$, define the \textbf{uncertainty} of $A$ at the point $h$ as 
\begin{align*}
	\Delta _h(A)\coloneqq \|Ah-\langle Ah, h \rangle h \|=\sqrt{\|Ah\|^2-\langle Ah, h \rangle^2}. 
\end{align*}
In 1929, Robertson \cite{ROBERTSON} derived the following mathematical form of the uncertainty principle (also known as uncertainty relation) of Heisenberg derived in 1927 \cite{HEISENBERG} (English translation of 1927 original article by Heisenberg). Recall that for two linear operators $A:	\mathcal{D}(A)\to  \mathcal{H}$ and $B:	\mathcal{D}(B)\to  \mathcal{H}$, we define $[A,B] \coloneqq AB-BA$ and $\{A,B\}\coloneqq AB+BA$.
\begin{theorem} \cite{ROBERTSON, CASSIDY, HEISENBERG, VONNEUMANNBOOK, DEBNATHMIKUSINSKI, OZAWA} (\textbf{Heisenberg-Robertson Uncertainty Principle}) \label{HRT}
Let  $A:	\mathcal{D}(A)\to  \mathcal{H}$ and $B:	\mathcal{D}(B)\to  \mathcal{H}$  be self-adjoint operators. Then for all $h \in \mathcal{D}(AB)\cap  \mathcal{D}(BA)$ with $\|h\|=1$, we have 
\begin{align}\label{HR}
 \frac{1}{2} \left(\Delta _h(A)^2+	\Delta _h(B)^2\right)\geq \frac{1}{4} \left(\Delta _h(A)+	\Delta _h(B)\right)^2 \geq  \Delta _h(A)	\Delta _h(B)   \geq  \frac{1}{2}|\langle [A,B]h, h \rangle |.
\end{align}
\end{theorem}
In 1930,  Schrodinger improved Inequality (\ref{HR}) \cite{SCHRODINGER} (English translation of 1930 original article by Schrodinger). 
\begin{theorem} \cite{SCHRODINGER}
	(\textbf{Heisenberg-Robertson-Schrodinger  Uncertainty Principle}) \label{HRS}
	Let  $A:	\mathcal{D}(A)\to  \mathcal{H}$ and $B:	\mathcal{D}(B)\to  \mathcal{H}$  be self-adjoint operators. Then for all $h \in \mathcal{D}(AB)\cap  \mathcal{D}(BA)$ with $\|h\|=1$, we have 
	\begin{align*}
		\Delta _h(A)	\Delta _h(B)    \geq |\langle Ah, Bh \rangle-\langle Ah, h \rangle \langle Bh, h \rangle|
		=\frac{\sqrt{|\langle [A,B]h, h \rangle |^2+|\langle \{A,B\}h, h \rangle -2\langle Ah, h \rangle\langle Bh, h \rangle|^2}}{2}.
	\end{align*}	
\end{theorem}
Theorem \ref{HRS} promotes the following question.
\begin{question}\label{Q}
	What is the nonlinear (even Banach space) version of Theorem \ref{HRS}?
\end{question}
  In this short note,  we answer Question \ref{Q} by deriving an uncertainty principle for Lipschitz maps acting on subsets of Banach spaces.  We note that there is a Banach space version of uncertainty principle by Goh and Goodman \cite{GOHGOODMAN}  which differs from the results in this paper. It is interesting to note that the uncertainty principle of Game theory, derived by Sz\'{e}kely and Rizzo is nonlinear \cite{SZEKELYRIZZO}.

\section{Nonlinear Heisenberg-Robertson-Schrodinger Uncertainty Principle}

Let $\mathcal{X}$ be a Banach space and  $\mathcal{M}\subseteq  \mathcal{X}$ be a subset such that $0 \in \mathcal{M}$.  Recall that  the collection of all Lipschitz functions $f:  \mathcal{M}\to \mathbb{C}$  satisfying  $f(0)=0$, denoted by $ \mathcal{M}^\# $ is a Banach space \cite{WEAVER} w.r.t. the Lipschitz norm 
\begin{align*}
	\|f\|_{\text{Lip}_0(\mathcal{M}, \mathbb{C})}\coloneqq \sup_{x,y \in \mathcal{M}, x \neq y} \frac{|f(x)-f(y)|}{\|x-y\|}.
\end{align*}
Let  $\mathcal{M}\subseteq  \mathcal{X}$ be a subset such that $0 \in \mathcal{M}$. Let $A:\mathcal{M}\to  \mathcal{X}$ be a Lipschitz map such that $A(0)=0$.
Given $x\in \mathcal{M}$ and $f \in  \mathcal{X}^\# $ satisfying $f(x)=1$, we define \textbf{two  uncertainties} of  $A$ at $(x,f) \in \mathcal{M}\times  \mathcal{X}^\# $ as 
\begin{align*}
	&\Delta (A, x, f)\coloneqq \|Ax-f(Ax) x\|,\\
	&\nabla(f, A,x)\coloneqq \|fA-f(Ax) f\|_{\text{Lip}_0(\mathcal{M}, \mathbb{C})}.
\end{align*}
\begin{theorem} (\textbf{Nonlinear Heisenberg-Robertson-Schrodinger Uncertainty Principle}) \label{NHRS}
	Let  $\mathcal{X}$ be a Banach space and  $\mathcal{M}, \mathcal{N}\subseteq  \mathcal{X}$ be  subsets such that $0 \in \mathcal{M}\cap \mathcal{N}$. Let $A:\mathcal{M}\to  \mathcal{X}$, $B:\mathcal{N}\to  \mathcal{X}$ be  Lipschitz maps such that $A(0)=B(0)=0$.  Then  for all $x\in \mathcal{M}\cap \mathcal{N}$ and $f \in  \mathcal{X}^\# $ satisfying $f(x)=1$, we have 
	\begin{align*}
 \frac{1}{2} \left(\nabla(f, A,x)^2+\Delta (B, x, f)^2\right)&\geq \frac{1}{4} \left(\nabla(f, A,x)+	\Delta (B, x, f)\right)^2\\
 & \geq \nabla(f, A,x)\Delta (B, x, f)\\
 &\geq |f(ABx)-f(Ax)f(Bx)|.
	\end{align*}
\end{theorem}
\begin{proof}
	\begin{align*}
	\nabla(f, A,x)\Delta (B, x, f)&=\|fA-f(Ax) f\|_{\text{Lip}_0(\mathcal{M}, \mathbb{C})} \|Bx-f(Bx) x\|\\
	&\geq |[fA-f(Ax) f][Bx-f(Bx) x]|\\
	&=|f(ABx)-f(Bx)f(Ax)-f(Ax)f(Bx)+f(Ax)f(Bx)f(x)|\\
	&=|f(ABx)-f(Bx)f(Ax)-f(Ax)f(Bx)+f(Ax)f(Bx)\cdot 1|\\
	&=|f(ABx)-f(Ax)f(Bx)|.
\end{align*}	
\end{proof}
\begin{corollary}
	Let  $\mathcal{X}$ be a Banach space and  $\mathcal{M}, \mathcal{N}\subseteq  \mathcal{X}$ be  subsets such that $0 \in \mathcal{M}\cap \mathcal{N}$. Let $A:\mathcal{M}\to  \mathcal{X}$, $B:\mathcal{N}\to  \mathcal{X}$ be  Lipschitz maps such that $A(0)=B(0)=0$.  Then  for all $x\in \mathcal{M}\cap \mathcal{N}$ and $f \in  \mathcal{X}^\# $ satisfying $f(x)=1$, we have 	
	\begin{align*}
		\nabla(f, A,x)\Delta (B, x, f)+\|fB\|_{\text{Lip}_0(\mathcal{N}, \mathbb{C})}\Delta (A, x, f)\geq |f([A,B]x)|.
	\end{align*}
\end{corollary}
\begin{proof}
	Note that 
	\begin{align*}
\nabla(f, A,x)\Delta (B, x, f)&\geq	|f(ABx)-f(Ax)f(Bx)|\\
&=|f(ABx-BAx)+f(BAx)-f(Ax)f(Bx)|\\
	&=|f([A,B]x)+(fB)[Ax-f(Ax)x]|\\
	&\geq |f([A,B]x)|-|(fB)[Ax-f(Ax)x]|\\
	&\geq |f([A,B]x)|-\|fB\|_{\text{Lip}_0(\mathcal{N}, \mathbb{C})}\|Ax-f(Ax)x\|\\
	&=|f([A,B]x)|-\|fB\|_{\text{Lip}_0(\mathcal{N}, \mathbb{C})}\Delta (A, x, f).
	\end{align*}
\end{proof}
\begin{corollary}
	Let  $\mathcal{X}$ be a Banach space and  $\mathcal{M}, \mathcal{N}\subseteq  \mathcal{X}$ be  subsets such that $0 \in \mathcal{M}\cap \mathcal{N}$. Let $A:\mathcal{M}\to  \mathcal{X}$, $B:\mathcal{N}\to  \mathcal{X}$ be  Lipschitz maps such that $A(0)=B(0)=0$.  Then  for all $x\in \mathcal{M}\cap \mathcal{N}$ and $f \in  \mathcal{X}^\# $ satisfying $f(x)=1$, we have 	
	\begin{align*}
			\nabla(f, A,x)\Delta (B, x, f)+\|fB\|_{\text{Lip}_0(\mathcal{N}, \mathbb{C})}\Delta (A, x, -f)\geq |f(\{A,B\}x)|.
	\end{align*}
\end{corollary}
\begin{proof}
	Note that 
\begin{align*}
	\nabla(f, A,x)\Delta (B, x, f)&\geq	|f(ABx)-f(Ax)f(Bx)|\\
	&=|f(ABx+BAx)-f(BAx)-f(Ax)f(Bx)|\\
	&=|f(\{A,B\}x)-(fB)[Ax+f(Ax)x]|\\
	&\geq |f(\{A,B\}x)|-|(fB)[Ax+f(Ax)x]|\\
	&\geq |f(\{A,B\}x)|-\|fB\|_{\text{Lip}_0(\mathcal{N}, \mathbb{C})}\|Ax-(-f)(Ax)x\|\\
	&=|f(\{A,B\}x)|-\|fB\|_{\text{Lip}_0(\mathcal{N}, \mathbb{C})}\Delta (A, x, -f).
\end{align*}	
\end{proof}
\begin{corollary} (\textbf{Functional Heisenberg-Robertson-Schrodinger Uncertainty Principle})
	 	Let  $\mathcal{X}$ be a Banach space with dual $\mathcal{X}^*$,  $A:	\mathcal{D}(A)\to  \mathcal{X}$ and $B:	\mathcal{D}(B)\to  \mathcal{X}$  be linear  operators.  	Then  for all $x \in \mathcal{D}(AB)\cap  \mathcal{D}(BA)$ and $f \in  \mathcal{X}^*$ satisfying $f(x)=1$, we have 
	 	\begin{align*}
	 		 \frac{1}{2} \left(\nabla(f, A,x)^2+\Delta (B, x, f)^2\right)&\geq \frac{1}{4} \left(\nabla(f, A,x)+	\Delta (B, x, f)\right)^2\\
	 		&\geq \nabla(f, A,x)\Delta (B, x, f)\\
	 		&\geq |f(ABx)-f(Ax)f(Bx)|.
	 	\end{align*}
\end{corollary}
\begin{corollary}
	Let  $\mathcal{X}$ be a Banach space,  $A:	\mathcal{D}(A)\to  \mathcal{X}$ and $B:	\mathcal{D}(B)\to  \mathcal{X}$  be linear  operators.  	Then  for all $x \in \mathcal{D}(AB)\cap  \mathcal{D}(BA)$ and $f \in  \mathcal{X}^*$ satisfying $f(x)=1$, we have 
\begin{align*}
\nabla(f, A,x)\Delta (B, x, f)+\|fB\|\Delta (A, x, f)\geq |f([A,B]x)|.
\end{align*}	
\end{corollary}
\begin{corollary}
Let  $\mathcal{X}$ be a Banach space,  $A:	\mathcal{D}(A)\to  \mathcal{X}$ and $B:	\mathcal{D}(B)\to  \mathcal{X}$  be linear  operators.  	Then  for all $x \in \mathcal{D}(AB)\cap  \mathcal{D}(BA)$ and $f \in  \mathcal{X}^*$ satisfying $f(x)=1$, we have 
\begin{align*}
	\nabla(f, A,x)\Delta (B, x, f)+\|fB\|\Delta (A, -x, f)&=\nabla(f, A,x)\Delta (B, x, f)+\|fB\|\Delta (A, x, -f)\\
	&\geq |f(\{A,B\}x)|.	
\end{align*}
\end{corollary}	
\begin{corollary}
Theorem \ref{HRS} follows from Theorem \ref{NHRS}.
\end{corollary}
\begin{proof}
Let $\mathcal{H}$ be a complex Hilbert space.	  $A:	\mathcal{D}(A)\to  \mathcal{H}$ and $B:	\mathcal{D}(B)\to  \mathcal{H}$  be self-adjoint operators. Let  $h \in \mathcal{D}(AB)\cap  \mathcal{D}(BA)$ with $\|h\|=1$. Define $\mathcal{X}\coloneqq \mathcal{H}$, $\mathcal{M}\coloneqq \mathcal{D}(A)$, $\mathcal{N}\coloneqq \mathcal{D}(B)$, $x\coloneqq h$ and 
\begin{align*}
	f:\mathcal{H}\ni u \mapsto f(u)\coloneqq\langle u, h \rangle \in \mathbb{C}.
\end{align*}
Then 
\begin{align*}
\Delta (B, x, f)=\Delta (B, h,f)= \|Bh-f(Bh) h\|= \|Bh-\langle Bh, h \rangle h \|=	\Delta _h(B), 
\end{align*}
\begin{align*}
	\nabla(f, A,x)&=\nabla(f, A,h)=\|fA-f(Ah) f\|=\|fA-\langle Ah, h \rangle f\|\\
	&=\sup_{u \in \mathcal{H}, \|u\|\leq 1}|f(Au)-\langle Ah, h \rangle f(u)|=\sup_{u \in \mathcal{H}, \|u\|\leq 1}|\langle Au, h \rangle -\langle Ah, h \rangle \langle u, h \rangle |\\
	&=\sup_{u \in \mathcal{H}, \|u\|\leq 1}|\langle u, Ah \rangle -\langle Ah, h \rangle \langle u, h \rangle |=\sup_{u \in \mathcal{H}, \|u\|\leq 1}|\langle u, Ah -\langle Ah, h \rangle h \rangle |\\
	&=\|Ah -\langle Ah, h \rangle h\|=	\Delta _h(A)
\end{align*}
and 
\begin{align*}
|f(ABx)-f(Ax)f(Bx)|&=|f(ABh)-f(Ah)f(Bh)|=|\langle ABh, h \rangle-\langle Ah, h \rangle \langle Bh, h \rangle|\\
&=|\langle Bh, Ah \rangle-\langle Ah, h \rangle \langle Bh, h \rangle|=|\langle Ah, Bh \rangle-\langle Ah, h \rangle \langle Bh, h \rangle|.
\end{align*}
\end{proof}
\section{Acknowledgments}
This research was partially supported by the University of Warsaw Thematic Research Programme ``Quantum Symmetries".

 \bibliographystyle{plain}
 \bibliography{reference.bib}

\end{document}